\documentclass[a4paper, 11pt]{amsart}

\usepackage{dsfont}
\usepackage{amsmath, amsthm, amssymb, amscd}
\usepackage{fullpage}
\usepackage{xypic}
\usepackage{graphicx}
\usepackage{color}
\usepackage{mathtools}
\usepackage[usenames,dvipsnames]{xcolor}
\usepackage[latin1]{inputenc}
\usepackage{indentfirst}
\usepackage{epstopdf}
\usepackage{subfigure}
\usepackage{emptypage}
\usepackage{hyperref}
\hypersetup{
    colorlinks=true, %set true if you want colored links
    linktoc=all,     %set to all if you want both sections and subsections linked
    linkcolor=blue,  %choose some color if you want links to stand out
    citecolor=JungleGreen,
}

\newtheorem{theorem}{Theorem}%[section]
\newtheorem{corollary}[theorem]{Corollary}

\newtheorem{proposition}[theorem]{Proposition}
\newtheorem{lemma}[theorem]{Lemma}
\newtheorem{definition}[theorem]{Definition}

\newtheorem{theorem*}{Theorem}
\newtheorem{question*}[theorem*]{Question}
\newtheorem{conjecture*}[theorem*]{Conjecture}
\newtheorem{corollary*}[theorem*]{Corollary}
\newtheorem{theorem*e}{Teorema}
\newtheorem{question*e}[theorem*e]{Pregunta}
\newtheorem{conjecture*e}[theorem*e]{Conjetura}
\newtheorem{corollary*e}[theorem*e]{Corolario}

\newcommand{\Fix}{\mathrm{Fix}}

\renewcommand{\deg}{\mathrm{deg}}

\begin{document}

\title{On the growth rate inequality for self-maps of the sphere}
\author{H\'ector Barge}

\address{E.T.S. Ingenieros Inform\'aticos, Universidad Polit\'ecnica de Madrid, 28660 Madrid, Espa\~na}
\email{h.barge@upm.es}
\thanks{The authors are partially supported by the Spanish Ministerio de Ciencia e Innovaci\'on (grants PGC2018-098321-B-I00 and PID2021-126124NB-I00)}

\author{Luis Hern\'andez-Corbato}
\address{Facultad de Ciencias Matem\'aticas, Universidad Complutense de Madrid, 28040 Madrid, Espa\~na}
\email{luishcorbato@mat.ucm.es}


\begin{abstract}
Let $S^m = \{x_0^2 + x_1^2 + \cdots + x_m^2 = 1\}$ and $P = \{x_0 = x_1 = 0\} \cap S^m$. Suppose that $f$ is a self--map of $S^m$ such that $f^{-1}(P) = P$ and $|\deg(f_{|P})| < |\deg(f)|$. Then, the number of fixed points of $f^n$ grows at least exponentially with base $|d| > 1$, where $d = \deg(f)/\deg(f_{|P}) \in \mathbb Z$.
\end{abstract}

\maketitle

\section{Introduction}

In \cite{shubasterisque}, Shub raised the question on whether algebraic intersections numbers bound assymptotically from below geometrical intersection numbers for $C^1$ maps. For a self-map $f \colon M \to M$ on a manifold, a particular case of the previous question comes down to whether the number $\#\Fix(f^n)$ of points fixed by $f^n$ and the Lefschetz numbers $L(f^n)$ satisfy

\[
\limsup \frac{1}{n} \log(\#\Fix(f^n)) \ge \limsup \frac{1}{n} \log\,|L(f^n)|
\]

This inequality is known as the \emph{growth rate inequality} and bounds from below the base of the exponential growth rate of periodic points. If the sequence of Lefschetz numbers $L(f^n)$ is unbounded, Lefschetz-Dold theorem together with a result of Shub and Sullivan \cite{shubsullivan} (that states that the index sequence of a fixed point for a $C^1$ map is bounded) imply that there are infinitely many periodic points. The growth rate inequality appeared again as an open problem in the Proceedings of the ICM 2006 \cite{shubICM}. It is wide open in dimensions greater than 1.

The Lefschetz number of a self-map on a sphere can be computed from the topological degree as $L(f) = 1 \pm \deg(f)$, so the growth rate inequality for self-maps of the sphere can be rewritten in terms of the degree as follows:

\begin{equation}\label{eq:growthratesphere}
\limsup_{n \to +\infty} \frac{1}{n} \log(\#\Fix(f^n)) \ge \log \,|\deg(f)|
\end{equation}

Shub's original paper \cite{shubasterisque} proved \eqref{eq:growthratesphere} for rational maps on $S^2$. However, even for a $C^1$ map of degree 2 on $S^2$, it is not known whether the growth rate inequality holds. Note that the $C^1$ assumption is crucial, a degree-2 north-south map on $S^2$ has only 2 fixed points and no other periodic point.

Recently, the growth rate inequality has been proved in several instances for maps on the sphere. In dimension 2, the sharp bound of $\log\,|\deg(f)|$ was obtained when $f$ preserves some singular foliations \cite{pughshub, misiurewiczlatitudinal, graffmonopole} or under hypothesis of dynamical nature \cite{indicesR2, uruguayosanillo, uruguayosS2, uruguayosbranchedcovering}. In higher dimensions the results are scarce. Weaker bounds for the growth rate of periodic points when the map preserves some foliation and some mild hypotheses is satisfied were obtained in \cite{grafflatitudinal, grafflongitudinal}. %\aviso{lo de graff siempre necesita hipotesis inductivas?}. 

In this article we prove a weak form of the growth rate inequality for maps on $S^m = \{x_0^2 + x_1^2 + \cdots + x_m^2 = 1\} \subset \mathbb R^{m+1}$. Suppose $f \colon S^m \to S^m$ leaves the codimension--2 sphere $P = \{x_0 = x_1 = 0\}$ completely invariant, that is, $f^{-1}(P) = P$. Then, the degree of $f$ is equal to the product of two factors: the degree of the restriction of $f$ to $P$ and a ``transversal'' degree denoted $d \in \mathbb Z$. The latter can be interpreted in terms of the action induced by $f$ on the homology group or the fundamental group of $S^m-P$.

\begin{theorem}\label{thm:intro}
Let $f \colon S^m \to S^m$ be a map such that $f^{-1}(P) = P$ and $\deg(f) \neq 0$ and let $d = \deg(f)/\deg(f_{|P})$. Then,
$\#\Fix(f^n) + \#\Fix(f^n_{|P}) \ge |d^n - 1|$. In particular,

\[
\liminf_{n \to +\infty} \frac{1}{n} \log(\#\Fix(f^n)) \ge \log|d|
\]
\end{theorem}

%The proof of the theorem uses mainly elementary topological arguments.
%We find fixed points on the universal cover of $S^m-P$ for most of the lifts of $f$ using topological degree theory. The argument requires a local analysis of the map at $P$.
%The biggest difficulty that has to be overcome is that critical fixed points in $P$ are not necessarily attractors, they may have repelling sectors. %As a consequence of this feature, we have to add the term $\#\Fix(f^n_{|P})$ on the first inequality of the statement of the theorem.

In dimension $m = 2$, $P$ is a 0--sphere and $\mathrm{deg}(f_{|P})$ can only take the values $-1, 0, 1$. We deduce the following corollary from Theorem \ref{thm:intro}, that was previously stated in \cite{uruguayosS2}.
%As an straightforward corollary of Theorem \ref{thm:intro} we deduce that growth rate inequality holds in $S^2$ as long as there are two points with only one preimage.
%\aviso{esto lo tienen los uruguayos como Corolario 5.2 en \cite{uruguayosS2} pero dicen que entonces los dos puntos criticos son atractores y no me lo creo}

\begin{corollary}
Suppose $f \colon S^2 \to S^2$ and there are two points $\{p,p'\}$ such that $f^{-1}(\{p,p'\}) = \{p,p'\}$. Then,
\[
\liminf_{n \to + \infty} \frac{1}{n} \log(\#\Fix(f^n)) \ge \log \,|\deg(f)|
\]
\end{corollary}

In higher dimensions we obtain a weak bound for the growth rate:

\begin{corollary}
Under the hypothesis of Theorem \ref{thm:intro}, if $m = 3$ or, more generally, if the growth rate inequality holds in $S^{m-2}$ then
\[
\liminf_{n \to +\infty} \frac{1}{n} \#\Fix(f^n) \ge \frac{1}{2}\log\,|\deg(f)|
\]
\end{corollary}
The result follows from the first inequality in Theorem \ref{thm:intro} and the fact that $\mathrm{max}\{|d|, |\deg(f_{|P})|\} \ge \sqrt{|\deg(f)|}$.

Theorem \ref{thm:intro} is deduced from Theorem \ref{thm:fixedpointslift}, which is stated in the ensuing section.
%The proof is mainly topological and is contained in the last section. We find fixed points on the universal cover of $S^m-P$ for most of the lifts of the map using topological degree theory. The definition of topological degree and some basic results are reviewed in Section \ref{sec:degree}. A local analysis of the map in the normal direction to $P$ is carried out in Section \cite{sec:analysis}. The biggest challenge in the proof of Theorem \ref{thm:fixedpointslift} is that critical fixed points in $P$ do not behave necessarily as attractors in the 2-dimensional normal direction, they may have repelling sectors. For this reason, the topological degree argument does not work for all lifts of the map.
The proof is contained in the last section of the article. It requires a detailed local analysis of the map in the normal direction to $P$, presented in Section \ref{sec:analysis}, and uses an argument from topological degree theory, which is quickly reviewed in Section \ref{sec:degree}.

%Elementary results from topological degree theory are reviewed in Section \ref{sec:degree} and the analytical description at $P$ is presented in Section \ref{sec:analysis}. The last section contains the proof of Theorem \ref{thm:fixedpointslift}.

%\aviso{Se puede estirar el chicle de los corolarios, pero yo no lo haria}
%Another corollary is that if $f_{|P}$ is a homeomorphism, the growth rate inequality holds.

\section{Setting}\label{sec:prelim}

Let $S = S^m = \{(x_0, \ldots, x_m): x_0^2 + \cdots + x_m^2 = 1\}$ be the standard $m$--sphere in $\mathbb R^{m+1}$ and $P = \{x_0 = x_1 = 0\} \cap 
S$ be an $m-2$--dimensional sphere which we will refer to as the \emph{polar sphere}. The complement $S - P$ is diffeomorphic to $S^1 \times D^{m-1}$, where $D^{m-1}$ denotes the $(m-1)$--dimensional open unit disk, via the map

\begin{equation}\label{eq:coordinatesS-P}
(x_0, \ldots, x_m)  \longmapsto  \left( \displaystyle \frac{(x_0, x_1)}{||(x_0, x_1)||}, (x_2, \ldots, x_m) \right)
\end{equation}

On the other hand, $P' = \{x_2 = \cdots = x_m = 0\} \cap S$ is a $1$--sphere such that $S$ is the join of $P$ and $P'$. The set $S - P'$ is diffeomorphic to $D^2 \times S^{m-2}$ by

% a system of neighborhoods of $P$ in $S$ is given by $W_{\epsilon} = \{x_0^2 + x_1^2 < \epsilon\} \cap S$. They can be given coordinates in $D^2(\epsilon) \times P$ by

\begin{equation}\label{eq:coordinatesS-P'}
(x_0, \ldots, x_m) \longmapsto  \left( (x_0, x_1),  \displaystyle \frac{(x_2, \ldots, x_m)}{||(x_2, \ldots, x_m)||} \right)
\end{equation}

Equations \eqref{eq:coordinatesS-P} and \eqref{eq:coordinatesS-P'} define coordinate charts for $S-P$ and $S-P'$, respectively. If we replace $(x_2, \ldots, x_m)$ by $(0,0,x_2, \ldots, x_m)$ in \eqref{eq:coordinatesS-P'}, we obtain a diffeomorphism between $S-P'$ and $D^2 \times P$. This description will be used extensively throughout this note, as it provides a product structure for the system of neighborhoods of $P$ defined by the inequalities $x_0^2 + x_1^2 < r$ for $0 < r < 1$: they are diffeomorphic to $D^2(r) \times P$.

Similarly, the inequalities $x_0^2 + x_1^2 > 1-r'$, $0 < r' < 1$, define neighborhoods of $P'$ in $S - P$ diffeomorphic by \eqref{eq:coordinatesS-P} to $S^1 \times D^{m-1}(r')$. Note that the radial coordinates of $D^2$ and $D^{m-1}$ are related by $r + r' = 1$. For any $0 < r < 1$,
\begin{equation}\label{eq:decompositionS}
S \enskip \cong \enskip D^2(r) \times P \enskip \sqcup \enskip S^1 \times \overline{D}^{m-1}(1-r)
\end{equation}

Evidently, the fundamental group of $S - P$ is isomorphic to $\mathbb Z$. Using the coordinates from \eqref{eq:coordinatesS-P'}, for any $r \in (0,1)$ and $p \in P$, it is easy to see that $\pi_1(S-P)$ is generated by a loop $\gamma$ that makes a positive turn around the origin in the 2-disk $D^2(r) \times \{p\}$. Let us consider the lift of $S-P$ that trivializes $[\gamma]$, $\mathrm{pr} \colon \widetilde{S-P} \to S-P$. 

In our results, $f$ is a self--map of $S$ for which $P$ is completely invariant, $f^{-1}(P) = P$.
Since $S-P$ is invariant under $f$, $f$ can be lifted to $\widetilde{S - P}$. Let $\tau$ be a generator of the group of deck transformations of the cover. Any lift $F$ of $f$ satisfies

\begin{equation}\label{eq:liftdegree}
F \, \tau = \tau^d \, F
\end{equation}

\noindent
where $d \in \mathbb Z$ is the \emph{transversal degree} defined by $f_*[\gamma] = d\cdot[\gamma]$ in $\pi_1(S-P)$ (see \eqref{eq:fundamentalgroup} later). As we will prove in Subsection \ref{subsec:degdecomp}, $\deg(f) = d\cdot \deg(f_{|P})$.

A fixed point of $F$ projects onto a fixed point of $f$ in $S - P$. The following lemma is a consequence of \eqref{eq:liftdegree} and establishes a relation between fixed points of different lifts. It is standard in Nielsen theory (cf. \cite[Lemma~4.1.10]{marzantowicz}).

\begin{lemma}\label{lem:nielsen}
If two different lifts $F, G = \tau^m F$ of $f$ satisfy $\mathrm{pr}(\Fix(F)) \cap \mathrm{pr}(\Fix(G)) \neq \emptyset$, then $d \neq 1$ and $|d-1|$ divides $m$.
\end{lemma}

By the previous lemma, we can bound from below $\#\{\Fix(f) \cap (S - P)\}$ by the number of lifts of $f$ among $\{F, \tau F, \ldots, \tau^{|d-1|-1} F\}$ that have fixed points.

\begin{theorem}\label{thm:fixedpointslift}
Suppose that $d \neq 0, 1$. 
There are at most $2\, \#\Fix(f_{|P})$ fixed point free maps among $\{F, \tau F, \ldots, \tau^{|d-1|-1} F\}$.
%Every map in $\{F, \tau F, \ldots, \tau^{|d-1|-1} F\}$, except for at most $2\, \#\Fix(f_{|P})$ maps, has a fixed point.
\end{theorem}

Theorem \ref{thm:fixedpointslift} and Lemma \ref{lem:nielsen} immediately imply that

\[
\#\{\Fix(f) \cap (S - P)\} \ge |d - 1| - 2\, \#\{\Fix(f_{|P})\} 
\]

\noindent Replace $f$ by $f^n$ in this inequality to deduce Theorem \ref{thm:intro}. Note that the transversal degree associated to $f^n$ is $d^n$.

%\noindent and Theorem \ref{thm:intro} follows automatically ($d$ cannot vanish because $\deg(f) \neq 0$).

\section{Topological degree}\label{sec:degree}

The topological degree of a map $g \colon M \to L$ between closed connected and oriented manifolds of dimension $n \ge 1$ roughly counts with multiplicity the number of preimages of a point. For a complete account on degree theory we refer to \cite{outereloruiz}. We give below a precise definition of the degree in algebraic topological terms. Recall that the orientation of a closed connected manifold  $M$ is given by a fundamental class $[M]$, i.e. a generator of the reduced homology group $\widetilde H_n(M)$. Reduced and unreduced homology groups only differ at dimension 0, the reason to choose here the reduced groups shall become clear later. The image $[M]_x$ of the fundamental class $[M]$ under the projection $\widetilde H_n(M) \to H_n(M, M-\{x\})$ is the local orientation at $x \in M$, and it is a generator of $H_n(M, M-\{x\})$. 

\begin{definition}
Let $g \colon M \to L$ be a map between $n$--dimensional closed orientable manifolds such that $\widetilde H_n(M) \cong \widetilde H_n(L) \cong \mathbb Z$ and $[M], [L]$ be fundamental classes of $M, L$, respectively. The degree of $g$, denoted $\mathrm{deg}(g)$, is the integer that satisfies
\begin{equation}\label{eq:defdegree}
g_*([M]) = \mathrm{deg}(g) \cdot [L],
\end{equation}
\end{definition}

The reason to choose the hypothesis $\widetilde H_n(M) \cong \mathbb Z$ in the definition is that it is satisfied by closed connected oriented manifolds of dimension $n \ge 1$ (spaces that appear in the standard definition of topological degree) and also by 0--spheres (such as the polar sphere in $S^2$, that is relevant in this paper). Incidentally, note that the degree of a map between 0--spheres can only take the values $-1, 0, 1$.

By duality, the degree can be alternatively defined using reduced cohomology groups. If $\omega_M, \omega_L$ are generators of the reduced $n$-th cohomology group of $M, L$, we have that $g^*(\omega_L) = \mathrm{deg}(g)\, \omega_M$.

\subsection{Decomposition of the degree}\label{subsec:degdecomp}

Suppose now that $f \colon M \to M$ is a self-map and $N$ is a completely invariant submanifold (i.e. $f^{-1}(N)=N$). Under some topological hypothesis, the degree of $f$ is equal to the product of two factors: the degree of the restriction of $f$ to $N$ and an integer $d$ that accounts for the winding around $N$ in the transversal direction, which we call \emph{transversal} degree. %\aviso{reducir y meter tras el lema?}

\begin{equation}\label{eq:degreedecomposition}
\mathrm{deg}(f) = d \cdot \mathrm{deg}(f_{|N})
\end{equation}

\begin{lemma}\label{lem:degdecomp}
Let $M$ be a closed connected and orientable $n$-dimensional manifold and $f:M\rightarrow M$ a continuous map with $\deg(f)\neq 0$. Suppose that $N\subset M$ is a closed orientable submanifold of codimension $k\geq 2$ that is completely invariant, i.e. $f^{-1}(N)=N$, and $\widetilde H^{n-k}(N) \cong \mathbb Z$. Then $\deg(f_{|N})$ divides $\deg(f)$. Moreover, if $H_{k}(M)\cong 0$ or if $k = n$ there is a non trivial class $\beta \in H_{k-1}(M - N)$ such that $f_*(\beta) = d \cdot \beta$, where $d=\deg(f)\, / \,\deg(f_{|N})$.
\end{lemma}

\begin{proof}
Observe that, since $N$ is completely invariant, the map $f$ can be considered as a map of pairs $f:(M,M-N)\rightarrow(M,M-N)$. Let $[M]\in H_n(M)$ a fundamental class in $M$. The homomorphism
\begin{equation}\label{eq:duality}
\frown[M]:H^{n-k}(N)\rightarrow H_k(M,M-N)
\end{equation}
consisting in capping each (unreduced) cohomology class in $H^{n-k}(N)$ with the fundamental class $[M]$ is an isomorphism (see \cite[Theorem~8.3, pg. 351]{bredon}). %también Hatcher Prop 3.46 p.256

Let $\omega$ be a generator of the reduced cohomology group $\widetilde H^{n-k}(N)$ (note that $\omega$ also generates $H^{n-k}(N)$ unless $k = n$). Then, by naturality of the cap product \cite[Theorem~5.2, pg. 336]{bredon} it follows %Hatcher p.241
\begin{equation}
f_*(f_{|N}^*(\omega)\frown[M])=\omega\frown f_*([M]).
\label{eq:cap}
\end{equation}
If we examine the left-hand side of \eqref{eq:cap} we get
\begin{equation}
f_*(f_{|N}^*(\omega)\frown[M])=f_*((\deg(f_{|N})\cdot\omega)\frown[M])=\deg(f_{|N})\cdot f_*(\omega\frown [M]).
\label{eq:cap1}
\end{equation}
On the other hand, 
\begin{equation}
\omega\frown f_*([M])=\omega\frown(\deg(f)\cdot[M])=\deg(f)\cdot(\omega\frown [M]).
\label{eq:cap2}
\end{equation}
Hence, from \eqref{eq:cap}, \eqref{eq:cap1} and \eqref{eq:cap2} it follows that $\deg(f_{|N})\mid\deg(f)$ and the quotient is an integer $d$ that satisfies $f_*(\omega \frown [M]) = d \cdot (\omega \frown [M])$ in $H_k(M, M - N)$.

The second statement follows immediately from the naturality of the long exact sequence of homology of $(M, M-N)$ in the case $H_k(M)$ is trivial. We can take $\beta$ to be the image of $\omega \frown [M]$ by the boundary morphism $H_k(M, M-N) \to H_{k-1}(M-N)$. In fact, by exactness the existence of $\beta$ is guaranteed as long as $\omega \frown [M]$ does not belong to the image of $p \colon H_k(M) \to H_k(M, M-N)$. In the case $k = n$, $N = \{x,y\}$ is a 0--sphere, $H_k(M)$ is generated by the fundamental class $[M]$ and $p([M]) = [M]_x + [M]_y$. Then, the preimage of $p([M])$ under the duality isomorphism \eqref{eq:duality} is the (unreduced) 0--cohomology class represented by the constant map on $N$ equal to 1 and, in particular, is different from $\omega$. Then, $\omega \frown [M] \notin \mathrm{im}(p)$ and the result follows.

\end{proof}

Heuristically, the transversal degree $d$ counts how many times the image of the boundary of a small neighborhood of $N$ wraps around $N$.
The interpretation is much clearer in the case $M = S = S^m$ and $N = P$, the polar sphere of codimension $k = 2$. Note that $H_2(S)$ is trivial for all $m > 2$ and if $m = 2 = k$, $P$ is a 0--sphere and the lemma still applies. From \eqref{eq:coordinatesS-P} we get that $H_1(S-P) \cong \mathbb Z$ and we deduce that
\[
(f_{|S-P})_* \colon H_1(S-P) \to H_1(S-P)
\]
is conjugate to the multiplication by $d$ in $\mathbb Z$. Evidently, the same description applies to the action induced in the fundamental group of $S-P$ as well: if $\gamma$ is a loop that generates $\pi_1(S-P)$ then 
\begin{equation}\label{eq:fundamentalgroup}
f_*[\gamma] = d\cdot[\gamma].
\end{equation}

\subsection{Vector fields and fixed points}

The final step in the proof of Theorem \ref{thm:fixedpointslift} uses an argument from topological degree theory. In order to keep the article self-contained, we formulate and prove the elementary results which are needed.

Let $U$ be an open subset of $\mathbb R^m$ and $B \subset U$ be diffeomorphic to $\overline{D}^m$. Any non-singular vector field $v$ on $\partial B$ defines a map $j \colon \partial B \to S^{m-1}$ by $j_v(x) = v(x)/||v(x)||$.

\begin{lemma}\label{lem:degreevectorfields}
\begin{itemize}
\item[(i)] If $v$ points inwards $B$ then $j_v$ is not nulhomotopic (i.e., not homotopic to the constant map).
\item[(ii)] If $v, w$ never point to the same direction (that is, $j_v(x) \neq j_w(x)$ for all $x$) and $j_v$ is not nulhomotopic then $j_w$ is not nulhomotopic.
\end{itemize}
Further, suppose that $\partial B$ is decomposed as the union of two hemispheres $E_+,E_-$, that is $(E_+, E_-)$ is diffeomorphic to $(H_+,H_-)$, where $H_+$ and $H_-$ denote the uppper and lower hemisphere of $S^{m-1}$.
\begin{itemize}
\item[(iii)] If $v$ points inwards on $E_+\cap E_-$, $j_v(E_+) \subset H_+$ and $j_v(E_-) \subset H_-$ then $j_v$ is not nulhomotopic.
\end{itemize}
\end{lemma}
\begin{proof}
(i) Clearly, $j_v$ is conjugate to a self-map of $S^{m-1}$ that is homotopic to the antipodal map. The conclusion follows from the fact that the antipodal map on $S^{m-1}$ is not nulhomotopic (otherwise, we could construct an homotopy from the identity map to a constant map by composing with the antipodal map).

(ii) $j_w$ is homotopic to $j_{-v}$. We can now use an argument as in (i) to conclude.

(iii) If $\sigma$ is the reflection through the equator on $S^{m-1}$, $\sigma \circ j_v$ is conjugate to the antipodal map. Again, we conclude that it is not nulhomotopic.
\end{proof}

Given a map $h \colon U \to \mathbb R^m$, we define a vector field $v_h(x) = h(x) - x$. Singularities of $v_h$ correspond to fixed points of $h$, so when we work with $j_{v_h}$ we tacitly assume $h$ has no fixed points on the boundary of $B$. One of the central ideas of topological degree theory is that it is possible to detect fixed points of $h$ inside $B$ just by studying $v_h$ on $\partial B$ or, more precisely, the homotopy class of $j_{v_h}$. Indeed, if $\Fix(h) \cap B = \emptyset$ then $v_h$ has no singularities in $B$ and, using a foliation of $B$ by spheres that converge to an interior point, it is possible to construct an homotopy from $j_{v_h}$ to a constant map. In other words,

\begin{lemma}\label{lem:degreemap}
If $j_{v_h}$ is not nulhomotopic, there exists a fixed point of $h$ in $B$.
\end{lemma}

Let us point out that one of the first results in topological degree theory is that the reverse implication is true up to homotopy. If $j_{v_h}$ is nulhomotopic, it is possible to construct an homotopy between $h$ and $h'$ relative to $\partial B$ such that $h'$ has no fixed points in $B$.

\section{Local analysis at fixed points in $P$}\label{sec:analysis}

Recall the setting from Section \ref{sec:prelim}.
%Recall from Section \ref{sec:prelim} that $P$ is a codimension--2 submanifold of $S$ that is completely invariant under $f$.
$P$ has a basis of neighborhoods diffeomorphic by \eqref{eq:coordinatesS-P'} to $D^2(r) \times S^{m-2}$. The map $f$ induces, by projection onto the first factor, a dynamics in the 2-dimensional normal direction around $P$. We obtain a family of $C^1$ maps $f_p \colon D^2(s) \to D^2$, $p \in P$, for some fixed small $s > 0$. %\aviso{redundante con la intro de 4.2?}

The smoothness of $f$ poses a restriction on the behavior of $f_p$ for a fixed point $p \in P$ because of the following reason: a $C^1$ map is injective in a neighborhood of a repelling fixed point. This fact follows from the inverse function theorem and the fact that the repelling condition implies that the eigenvalues of the jacobian matrix at the fixed point lie outside the unit disk and, in particular, away from zero. We shall prove later that if any $f_p$ is injective then the transversal degree satisfies $|d| \le 1$ and Theorem \ref{thm:intro} becomes trivial. Accordingly, we focus on the case $|d| > 1$ in which $f_p$ is not injective for any $p \in P$ and, in particular, the Jacobian matrix $A_p$ of $f_p$ at the origin in the 2--dimensional normal direction is singular. Therefore, there are only two dynamically different cases stated in terms of the spectral radius of $A_p$, either it is smaller than 1 and the origin is an attractor for $f_p$ or it is greater or equal than 1 and there is an attracting cone region.

%the origin is not a repeller for any $f_p$, $p \in \mathrm{Fix}(f_{|P})$.

We proceed now to study the dynamics of planar maps such as $f_p$. Later, we apply the local picture to describe the behavior of $f$ in the normal direction to $P$.

\subsection{Planar results}

Suppose $g \in \mathbb R^2 \to \mathbb R^2$ is a $C^1$ map that fixes the origin, $g(0) = 0$. Denote by $A$ the Jacobian matrix of $g$ at $0$. By the definition of differentiability at the origin, for every $\epsilon > 0$ there exists $\delta > 0$ such that

\begin{align}\label{eq:diffR2}
\frac{||g(u) - Au||}{||u||} < \epsilon, \quad \quad \text{for all}\enskip u \enskip \text{such that}\enskip 0 < ||u|| < \delta
\end{align}
\noindent where we have used the identification $T_0\mathbb R^2 \cong \mathbb R^2$ and $||\cdot||$ is a norm in $\mathbb R^2$. Recall that all the norms in a finite dimensional vector space are equivalent. The spectral radius of $A$, $\rho(A)$, largely determines the behavior of $g$ in a neighborhood of $0$.

\begin{lemma}
For every $c > \rho(A)$ there exists a norm $||\cdot||$ in $\mathbb R^2$ such that

\centerline{
$||Au|| < c\,||u||$ for every $u \in \mathbb R^2 - \{0\}$.
}
\end{lemma}
\begin{proof}
If $A$ is diagonalizable over $\mathbb R$, we can take the $\ell^1$--norm associated to a basis $\mathcal B$ composed of eigenvectors, that is, $||u|| = ||(u_1,u_2)_{\mathcal B}|| := |u_1| + |u_2|$. If the eigenvalues of $A$ are not real, the $\ell^2$--norm associated to an orthogonal basis satisfies the conclusion. Finally, if the eigenvalues of $A$ are equal but $A$ is not diagonalizable, let $e_0$ be an eigenvector and $e_1 \neq 0$ not collinear to $e_0$. Then, we can take the $\ell^1$--norm associated to the basis $\{Ke_0, e_1\}$ for large enough $K > 0$.
\end{proof}

An immediate consequence of the previous lemma and \eqref{eq:diffR2} is that if $\rho(A) < 1$ then the origin is a local attractor for $g$.

\begin{corollary}\label{cor:spectrum<1attractor}
Suppose $\rho(A)< 1$ and let $\epsilon \in (0, 1-\rho(A))$, then there exists $\delta > 0$ and a norm in $\mathbb R^2$ such that $||g(u)|| < (1-\epsilon) ||u||$ whenever $0 < ||u|| < \delta$.
\end{corollary}

In the case there is an eigenvalue $\lambda$ with $|\lambda|\ge 1$, we can locate the region where the inequality $||g(u)|| < ||u||$ does not hold.

\begin{lemma}\label{lem:sectorplano}
Suppose the eigenvalues of $A$ are $\{0, \lambda\}$ with $|\lambda| \ge 1$. Denote $\mathcal B = \{e_0, e_{\lambda}\}$ a basis composed of eigenvectors, $||\cdot||$ the $\ell^1$--norm associated to $\mathcal B$ and 

\[
C(\alpha) = \{u = (u_0, u_{\lambda})_{\mathcal B} \in \mathbb R^2-\{0\}: |u_0/u_{\lambda}| <\alpha\}.
\]
\noindent
For every $\epsilon\in (0,1/2)$ there exists $\delta > 0$ such that if $0 < ||u|| < \delta$ and
\begin{itemize}
\item[(i)] if $u \notin C\left(\frac{|\lambda|+2\epsilon-1}{1-2\epsilon}\right)$ then $||g(u)|| < (1-\epsilon)||u||$.
\item[(ii)] if $u \in C\left( \frac{|\lambda|-3\epsilon}{3\epsilon} \right)$ then $|(g(u))_{\lambda}| > \epsilon||u||$ and, in particular, $g(u) \notin \langle e_0 \rangle$. ($\langle v \rangle$ denotes the subspace spanned by $v$)
\end{itemize}
\end{lemma}
\begin{proof}
Apply \eqref{eq:diffR2} to $||\cdot||$ and $\epsilon$. For (i), if $u = (u_0, u_{\lambda})_{\mathcal B}$ with $||u|| < \delta$

\[
(1 - \epsilon)||u|| = (1-\epsilon)|u_0|+(1-\epsilon)|u_{\lambda}| \ge \epsilon |u_0| + (|\lambda| + \epsilon) |u_{\lambda}| = ||Au|| + \epsilon ||u|| > ||g(u)||
\]

\noindent
To deduce (ii) we use that $||g(u)||+\epsilon||u|| \ge ||Au||$:

\[
|(g(u))_{\lambda}| = ||g(u)||-|(g(u))_0| \ge |\lambda||u_{\lambda}| - 2\epsilon ||u|| = (|\lambda|- 2\epsilon)|u_{\lambda}| - 2\epsilon|u_0| > \epsilon||u|| > 0.
\]

\end{proof}

\subsection{Analysis in the normal direction}

Since $f(P) = P$, $f$ restricts to a continuous map

\centerline{$f \colon D^2(s) \times P \longrightarrow D^2 \times P$}
\noindent for some $s > 0$, where we extensively use the coordinates introduced in \eqref{eq:coordinatesS-P'}. The Jacobian matrix of $f$ in a point $p \in P$ has the following form:

\centerline{
$Jf_p = \begin{pmatrix}
A_p & 0 \\
\ast & (Jf_{|P})_p \\
\end{pmatrix}$
}
\noindent
where the basis for the tangent space at $p$ is ordered according to the local product structure around $P$: first the (2-dimensional) normal space to $P$ and then the ($(m-2)$--dimensional) tangent space to $P$. Alternatively, $A_p$ can be defined as the Jacobian at $0$ of the following composition

\begin{align}\label{def:perp}
f_p \colon D^2(s) \to D^2(s) \times \{p\} \hookrightarrow D^2(s) \times P  \xrightarrow{\phantom{a}f\phantom{a}} D^2 \times P \xrightarrow{\mathrm{proj}} D^2
\end{align}

\begin{lemma}
If $|d| > 1$ then $A_p$ is singular for every $p \in P$.
\end{lemma}
\begin{proof}
Suppose that $A_p$ is regular. Then, $f_p$ restricts to a diffeomorphism between $D^2(r)$ and $V = f_p(D^2(r))$, for small $r > 0$. In particular, if $\gamma$ is a generator of $\pi_1(D^2(r) - \{0\})$, then $f_p(\gamma)$ is a generator of $\pi_1(V - \{0\})$.
Think of $\gamma$ as a loop in $(D^2(r)-\{0\}) \times\{p\}$ and choose it small so that $f(\gamma) \subset D^2 \times U_{f(p)}$, where $U_{f(p)}$ is contractible in $P$. It follows that both $\gamma$ and $f(\gamma)$ generate $\pi_1(S-P)$ and, by \eqref{eq:fundamentalgroup}, we deduce that $d = \pm 1$, a contradiction.
\end{proof}

Therefore, either Corollary \ref{cor:spectrum<1attractor} or Lemma \ref{lem:sectorplano} apply to $f_p$ when $|d|>1$. For a given $p \in P$, we extend the results by continuity to $f_q$ for $q$ close to $p$. 

\begin{proposition}\label{prop:resumen} Suppose that $|d|>1$.
For every $p \in P$, there exists a neighborhood $D^2(\delta) \times U_p$ of $p$ in $S$ and a norm $||\cdot||$ in $\mathbb R^2$ such that:

\begin{itemize}
\item[(i)] If $\rho(A_p) < 1$, for all $q \in U_p$ and all $u \in D^2(\delta)$,

\centerline{ $||f_q(u)|| \le ||u||$.}

\item[(ii)] Otherwise, there is a basis $\{e_0, e_1\}$ of $\mathbb R^2$ and $\alpha \in \mathbb R^+$ such that for every $q \in U_p$
\begin{itemize}
\item if $u \notin C(\alpha)$ and $u \in D^2(\delta)$ then $||f_q(u)|| \le ||u||$.
\item $f_q(C(\alpha) \cap D^2(\delta)) \cap \langle e_0 \rangle = \emptyset$. \end{itemize}

\end{itemize}

\end{proposition}

\begin{proof}
Let $||\cdot||$ be the norm from Corollary \ref{cor:spectrum<1attractor} or Lemma \ref{lem:sectorplano} depending on which alternative, (i) or (ii), applies.

For (i), apply Corollary \ref{cor:spectrum<1attractor} to any $\epsilon \in (0, 1-\rho(A))$ to obtain $||\cdot||$ and $\delta$ such that $||f_p(u)|| \le (1-\epsilon)||u||$ holds whenever $||u||<\delta$. In order to extend the conclusion to $f_q$, for $q$ close to $p$, we use the smoothness of $f$.

Since $f_q(u) = \left(\int_0^1 Df_q|_{tu}\, dt\right) u$, we have that

\begin{equation}\label{eq:acotacionC1}
||f_q(u) - f_p(u)|| \le \left( \int_0^1 ||Df_q|_{tu} - Df_p|_{tu}||\,dt\right) ||u|| < \gamma_q||u||
\end{equation}

\noindent
where $\gamma_q = \mathrm{max}_{v \in D^2(\delta)}||Df_q|_v - Df_p|_v||$. Since $f$ is $C^1$, $\gamma_q \to 0$ as $q \to p$. Let $U_p$ be a neighborhood of $p$ in $P$ such that $|\gamma_q| < \epsilon$ for all $q \in U_p$. Then we conclude that $||f_q(u)|| \le ||u||$ for all $q \in U_p$ and $||u|| < \delta$.

For (ii), denote $\lambda$ be the non--zero eigenvalue of $A_p$. Firstly, take $\epsilon > 0$ small enough so that

\centerline{
$\frac{|\lambda|+2\epsilon-1}{1-2\epsilon} < \frac{|\lambda|-3\epsilon}{3\epsilon} $
}

\noindent
and set $\alpha = \frac{|\lambda|+2\epsilon-1}{1-2\epsilon}$. Apply Lemma \ref{lem:sectorplano} to obtain $||\cdot||$, $\delta$ and $\{e_0, e_1\}$. The conclusions for $q = p$ follow immediately from the lemma. To extend the results to a neighborhood $U_p$ of $p$ we use \eqref{eq:acotacionC1} and proceed exactly as in (i). The argument above can be used verbatim to prove the first item. For the second item, $|(f_q(u))_{\lambda}| \ge |(f_p(u))_{\lambda}| - \epsilon ||u|| > 0$.
Finally, note that the norm from Lemma \ref{lem:sectorplano} may not be the $\ell^2$-norm so we might need to shrink $\delta$ to $\delta'$ so that the standard 2-disk $D^2(\delta')$ fits inside the disk defined by $||u||<\delta$.
\end{proof}

Below, in the proof of Theorem \ref{thm:fixedpointslift}, only the fixed points $p$ of $P$ for which the second alternative in Proposition \ref{prop:resumen} applies require special attention. The local description obtained above provides a cone $C(\alpha)$ above each $q$ close to $p$ that contains the repelling sector (if any) and whose image misses the attracting direction spanned by $e_0$.

\section{Proof of Theorem \ref{thm:fixedpointslift}}

Recall (see \eqref{eq:coordinatesS-P}) that $S - P$ is diffeomorphic to $S^1 \times D^{m-1}$. Consider the cover

\begin{equation}\label{eq:liftS-P}
\widetilde{S - P} \cong \mathbb R \times D^{m-1} \longrightarrow S^1 \times D^{m-1} \cong S-P
\end{equation}

\noindent
defined as the standard cover $\mathbb R \to S^1$, $t \mapsto e^{2\pi i t}$ in the first factor and as the identity in the second factor.

Our aim is to prove that except for a few cases, every lift of $f$ to $\widetilde{S-P}$ has a fixed point in $B_M = [-M, M] \times D^{m-1}(1-\delta)$ for large $M$ and small $\delta > 0$. The argument is based on topological degree theory.
The key observation is that, for most of the lifts $F$, the vector field $v_F(x) = F(x)-x$ in $\mathbb R \times D^{m-1}$ never points to the same direction as the coordinate vector field $\partial/\partial r(x)$, whose definition will be recalled next, on the boundary of $\mathbb R \times D^{m-1}(1-\delta)$. Then, we can apply Lemmas \ref{lem:degreevectorfields} and \ref{lem:degreemap} to conclude that $F$ has a fixed point inside $B_M$ for large $M > 0$.

From \eqref{eq:coordinatesS-P'}, we deduce that $S-P-P'$ is diffeomorphic to $D^2_0 \times P$, where the subscript in $D_0^2$ indicates that the disk is punctured at the origin. The lift of $S-P-P'$ to the cover \eqref{eq:liftS-P} is therefore diffeomorphic to $\mathbb R \times (0,1) \times P$, where the second factor, $(0,1)$, corresponds to the radial coordinate $r$ of $D^2_0$ and also to $1-r'$, where $r'$ is the radial coordinate of $D^{m-1}$ in \eqref{eq:liftS-P} (cf. the discussion before \eqref{eq:decompositionS}). The coordinate $r$ and, more precisely, the vector field $\partial/\partial r$ that it defines in the cover play a central role in the discussion. On the lateral face of the cylinder $B_M$, $\partial/\partial r$ points inwards. Note that, alternatively, if we use the coordinates from \eqref{eq:coordinatesS-P} instead of those of \eqref{eq:coordinatesS-P'}, we see that the lift of $S-P-P'$ is diffeomorphic to $\mathbb R \times D^{m-1}_0$.

%Let $\partial/\partial r$ be the coordinate vector field in $\widetilde{S-P} - P' \cong \mathbb R \times (0,1) \times P$ associated to the radial coordinate in normal bundle around $P$ (which corresponds to the factor $(0,1)$ in the product). 

The lift of the partition \eqref{eq:decompositionS} of $S-P$ for $r = \delta$ still displays a product structure:

\centerline{
$\widetilde{S - P} = \mathbb R \times (0,\delta) \times P \sqcup \mathbb R \times D^{m-1}(1-\delta)$
}
\noindent
The factor $\mathbb R$ corresponds to the angular coordinate in the normal bundle of $P$ in the first term and to the lift of $S^1$ (that parametrizes $P'$) in the second instance. % and also to the angular coordinate in the normal bundle of $P$.
The projection onto the first factor $\mathrm{pr} \colon \widetilde{S-P} \to \mathbb R$ conjugates a generator $\tau$ of the group of deck transformations of the cover and the translation by 1 in $\mathbb R$. 

Recall from the statement of Theorem \ref{thm:fixedpointslift} that $d \neq 0, 1$.
For an arbitrary lift $F$ of $f$, by \eqref{eq:liftdegree} we have that $F \tau = \tau^d F$ so

\centerline{
$\mathrm{pr}(x) = \mathrm{pr}(y) + 1 \quad \Rightarrow \quad \mathrm{pr}(F(x)) = \mathrm{pr}(F(y)) + d$.
}

\noindent
for all $x, y \in \widetilde{S-P}$. Therefore, for a fixed $\delta > 0$, if $M$ is sufficiently large and if $d > 1$ then

\begin{align}\label{eq:tapas}
\begin{split}
\mathrm{pr}(F(x)) &\le -M - 1 \quad \text{for every } x \in \{-M\} \times D^{m-1}(1 - \delta) \text{ and} \\
\mathrm{pr}(F(x)) &\ge M + 1 \quad \text{for every } x \in \{M\} \times D^{m-1}(1 - \delta)
\end{split}
\end{align}

\noindent
whereas if $d \le -1$ then

\begin{align}\label{eq:tapasdnegativo}
\begin{split}
\mathrm{pr}(F(x)) &\ge -M + 1 \quad \text{for every } x \in \{-M\} \times D^{m-1}(1 - \delta) \text{ and} \\
\mathrm{pr}(F(x)) &\le M - 1 \quad \text{for every } x \in \{M\} \times D^{m-1}(1 - \delta)
\end{split}
\end{align}

These inequalities imply that on the left and right faces (as in Figure \ref{fig:cilindro}) of the solid cylinder $B_M$ the vector field $v_F$ points outwards when $d > 1$ and inwards when $d \le -1$. 

Let us focus now on the lateral face of $B_M$ and suppose further that $d \neq -1$.
Let $p \in \mathrm{Fix}(f|_P)$ and consider the neighborhood $D^2(\delta) \times U_p$ of $p$ and the norm $|| \cdot ||$ from Proposition \ref{prop:resumen}. Denote $V_p$ the lift of $D^2_0(\delta) \times U_p$ to the cover (note that the disk is punctured) and suppose $v_F$ and $\partial/\partial r$ point to the same direction at $x \in V_p$. This implies that the projection of $x$ and $F(x)$ to $D^2 \times P$, have the form $(u, q)$ and $(au, q)$ for some $a > 1$ and $q \in U_p$. In particular, $f_q(u) = au$, which automatically implies $||f_q(u)|| > ||u||$. Thus, the second alternative of Proposition \ref{prop:resumen} applies to $p$, so there exists a basis $\{e_0, e_1\}$ and $\alpha \in \mathbb R^+$ such that $u \in C(\alpha) \cap D^2(\delta)$ and $f_q(C(\alpha)\cap D^2(\delta)) \cap \langle e_0\rangle = \emptyset$ for every $q \in U_p$. The last property can be stated equivalently as

\begin{equation}\label{eq:sectorimage}
f((C(\alpha) \cap D^2(\delta)) \times U_p) \cap (\langle e_0 \rangle \times U_p) = \emptyset
\end{equation}

We now lift these elements to the cover. The cone region $C(\alpha) \times U_p$ lifts to a sequence of domains $O_n = (I + \frac{n}{2}) \times (0,1) \times U_p$, indexed by $n \in \mathbb Z$, where we are now employing the coordinates of $\mathbb R \times (0,1) \times P$, the lift of $D_0^2 \times P$, and $I$ is an interval in $\mathbb R$ of length $< 1/2$. See Figure \ref{fig:cilindro}. Similarly, the strip $\langle e_0 \rangle \times U_p$ lifts to a sequence of strips $E_n = \{\widetilde e_0 + \frac{n}{2}\} \times (0,1) \times U_p$ for some $\widetilde e_0 \in \mathbb R$. Restricted to $V_p$, the strips $E_n$ and the domains $O_n$ are pairwise disjoint and are placed alternately. Evidently, $\tau(E_n) = E_{n+2}$ and $\tau(O_n) = O_{n+2}$.

Denote $O_n^{\delta} = O_n \cap (\mathbb R \times (0, \delta) \times U_p)$, the subset of $O_n$ defined by $0 < r < \delta$. The condition \eqref{eq:sectorimage} implies that $F(O_n^{\delta})$ does not meet $E_m$ for any $m \in \mathbb Z$. This imposes a serious restriction on the number of lifts for which the image of $O_n^{\delta}$ intersects itself.

\begin{lemma}
There are at most two elements $G$ of $\{F, \tau F, \ldots, \tau^{|d-1|-1} F\}$ such that

\centerline{
$G(O_n^{\delta}) \cap O_n \neq \emptyset$ \enskip for some $n \in \mathbb Z$.
}
\end{lemma}
\begin{proof}
Recall that for any lift $G$, $G (O_{n+2}) = G \tau(O_n) = \tau^d G(O_n)$ so if $G(O_n^{\delta})$ lies in between $E_k$ and $E_{k+1}$ then $G(O_{n+2}^{\delta})$ lies in between $E_{k+2d}$ and $E_{k+2d+1}$. Suppose that $G, G' = \tau^m G$ are two lifts of $f$ such that $G(O_n^{\delta}) \cap O_n \neq \emptyset$ and $G'(O_{n+2s}^{\delta}) \cap O_{n+2s} \neq \emptyset$ for some $s \neq 0$. It follows that $n + 2s = m + n + 2sd$, so $|d-1|$ divides $m$.

In sum, there is at most one lift in $\{F, \tau F, \ldots, \tau^{|d-1|-1} F\}$ such that the intersection in the statement is non-empty for some even $n$ and at most one lift such that the intersection is non-empty for some odd $n$.
\end{proof}

\begin{figure}
%\centerline{
\includegraphics[scale=0.9]{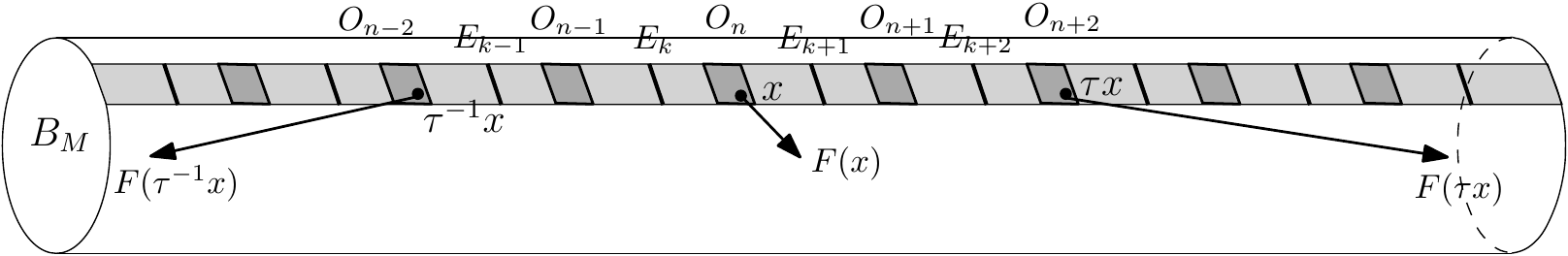}
\caption{
Picture of a piece of $B_M$ for $d = 2$. The darker pieces and thicker segments in the horizontal strip are the intersection of the domains $O_n$ and the strips $E_n$ with the lateral face of $B_M$, defined by $r = \delta$. Arrows illustrate $v_F$ at $\tau^{-1} x, x, \tau x$.
}
\label{fig:cilindro}
\end{figure}

Incidentally, note that the previous lemma trivially holds when $d = -1$. As a consequence of the discussion above we deduce:

\begin{lemma}\label{lem:entornopuntofijo}
Let $p \in \mathrm{Fix}(f_{|P})$ and $d \neq 0,1$. There exists $\delta_p > 0$, $U_p$ neighborhood of $p$ in $P$ such that there are at most two lifts $G$ among $\{F, \tau F, \ldots, \tau^{|d-1|-1} F\}$ for which the vector field $v_G(x) = G(x)-x$ points to the same direction as $\partial/\partial r$ at some point of $V_p$.
\end{lemma}

In the second part of the proof, we show that the lifts for which $v_G$ and $\partial/\partial r$ do not point to the same direction at $V_p$, for any $p \in \mathrm{Fix}(f_{|P})$, always have a fixed point on $B_M$. In view of the previous corollary this assertion concludes the proof.

 We have already described what happens in a neighborhood $\cup U_p$ of the set of fixed points in $P$. Since $f_{|P}$ has no fixed point on $P - \cup U_p$, by continuity, there exists $\delta_1$ such that every point in $D^2(\delta_1) \times (P - \cup U_p)$ is displaced tangentially to $P$ by $f$, that is, if $f((v, q)) = (f_q(v), q')$ then $q \neq q'$. As a consequence, for any lift $G$ of $f$, $v_G$ and $\partial/\partial r$ do not point to the same direction on $\mathbb R \times (0, \delta_1) \times (P - \cup U_p)$.

Take $\delta > 0$ smaller than all $\delta_p$ and $\delta_1$. The results from the previous paragraph and Lemma \ref{lem:entornopuntofijo} imply that there are at least $|d -1| - 2\#\mathrm{Fix}(f_{|P})$ lifts among $\{F, \tau F, \ldots, \tau^{|d-1|-1} F\}$ such that $v_G$ and $\partial/\partial r$ do not point to the same direction on the region $\{r = \delta\} = \mathbb R \times \{\delta\} \times P$. Say $G$ is one of them.

Take $M$ large enough so that \eqref{eq:tapas} holds for $G$ and $\delta$. Smooth out a small neighborhood of the edges of the cylinder $B_M$ to obtain a convex domain $B'_M$ diffeomorphic to a closed ball. We now define a vector field $w$ on $\partial B'_M$ to apply the results from Section \ref{sec:degree}.

\textbf{Case $d > 1$}. Let $w$ be the normal unitary vector on $\partial B'_M$ that point inwards. Note that $w$ coincides with $\partial/\partial r$ on the lateral face of $B'_M$. In the rest of $\partial B'_M$ the inequalities \eqref{eq:tapas} hold provided the smoothing region is small enough. By the choice of $G$ we conclude that $w$ never points to the same direction as $v_G$.

Since $B'_M$ is diffeomorphic to a ball and $w$ points inwards on its boundary, we can apply Lemma \ref{lem:degreevectorfields} (i) and (ii) to deduce that $j_{v_G}$ is not nulhomotopic. Then, Lemma \ref{lem:degreemap} concludes that $G$ has a fixed point inside $B'_M$, as desired.

\textbf{Case $d \le -1$}. Define $w$ as the unitary vector that points inwards on the lateral face of $B'_M$ and as the unitary vector that points outwards on the pieces of $\partial B'_M$ that are part of the left and right faces of $B_M$. Complete the definition of $w$ on $\partial B'_M$ by an interpolation that guarantees that in the smoothing region and close to the left face, $\{-M\} \times D^{m-1}(1-\delta)$, $w$ never points in the positive direction (increasing first coordinate), while close to the right face, $\{M\} \times D^{m-1}(1-\delta)$, $w$ never points in the negative direction (decreasing first coordinate). Again, by the choice of $G$, $w$ and $v_G$ never point to the same direction and, by Lemma \ref{lem:degreevectorfields} (iii) and (ii) and Lemma \ref{lem:degreemap} we conclude that $G$ has a fixed point in $B'_M$.

\bibliographystyle{plain}
\bibliography{biblioshub}

\end{document}